%% file: main.tex
\documentclass[11pt, oneside]{article}

\usepackage{amsmath, amssymb}
\usepackage{geometry}
\usepackage[dvips]{graphicx}
\usepackage{graphics}
\usepackage{epsfig}
\usepackage{psfrag, pstricks}

\usepackage{url}

\newtheorem{theorem}{Theorem}

\newtheorem{lemma}[theorem]{Lemma}

\newenvironment{proof}{{\noindent \textbf{\textit{Proof.}}}}
{\hfill $\Box$}

\long\def\symbolfootnote[#1]#2{\begingroup\def\thefootnote{\fnsymbol{footnote}}
\footnote[#1]{#2}\endgroup}

\begin{document}

\title{A Note on Schnyder's Theorem}
\author{Fidel Barrera-Cruz\footnote{fbarreracruz@uwaterloo.ca;
  Partially supported by CONACYT, Project 209395/304106.}
\\Penny Haxell\footnote{pehaxell@math.uwaterloo.ca; Partially
supported by NSERC.}\\
\medskip\\
Department of Combinatorics and Optimization\\
University of Waterloo\\
Waterloo, Ontario, Canada N2L 3G1}
\date{}

\maketitle

\begin{abstract}
We give an alternate proof of Schnyder's Theorem, that the incidence
poset of a graph $G$ has dimension at most three if and only if $G$ is
planar.
\end{abstract}



\noindent Order journal (July 2011) 28:221-226. Submitted: April 2010. Accepted: July 2010.

\noindent Final publication available at:
\url{http://dx.doi.org/10.1007/s11083-010-9167-z}

\begin{section}{Introduction}
An important theorem of Schnyder \cite{S} from 1989 relates two
different notions of dimension for graphs $G$, namely the
dimension of the incidence poset of $G$, and the minimum dimension in
which $G$ has a geometric realization. Schnyder's Theorem states that
the incidence poset of $G$ has dimension at most three if and only if
$G$ is planar. This work implies many useful results about planar
graphs (see e.g. \cite{S}) and motivated a significant body of
further work, for example \cite{BT1, BT2, CFL, F1, F2, FT, FZ, pom}.

Schnyder's proof developed a substantial amount of theory, involving
notions of \emph{normal labellings}, \emph{dual orders}, and
\emph{tree decompositions} among others, which itself contributes much
to the useful consequences of his theorem mentioned above. Each of the
two implications of Schnyder's Theorem can also be derived from other
works, for example \cite{BD, pom, DPP}. However,
we felt it could be useful to have a more direct proof, and
the aim of this note is to provide one.

\end{section}

\begin{section}{Basics}
Let $V$ be a finite set of vertices. We will say that a set
$R=\{<_1,<_2,<_3\}$ of three linear orders 
of $V$  is a \emph{standard representation} of $V$ if the
following conditions hold.
\begin{itemize}
\item ($R$ \emph{represents} $V$) The orders $<_1,<_2,<_3$ have empty
  intersection, in other words, for 
  every $x\not= y$ 
  in $V$ there exists $i\in\{1,2,3\}$ such that $x<_iy$.
\item ($R$ is \emph{standard}) For each $i$, the maximum element $a_i$
  of $<_i$ is among the 
  two smallest elements of $<_j$ for each $j\not= i$. 
\end{itemize}
We write $R=R(a_1,a_2,a_3)$ to indicate that $a_i$ is the maximum
element of $<_i$ for each $i$. We define a graph $\Sigma_2(R)$ with
vertex set $V$ as follows: the pair $xy$ of distinct vertices is an
edge of $\Sigma_2(R)$ 
if and only if for every $z\in
V\setminus\{x,y\}$ there exists $i\in\{1,2,3\}$ such
that $z>_ix$ and $z>_iy$. (The notation $\Sigma_2(R)$ reflects the
fact that this graph is part of the \emph{complex} $\Sigma(R)$ of $R$,
see e.g. \cite{pom}.)
Then Schnyder's Theorem can be formulated
as follows.

\begin{theorem}\label{sch}(Schnyder's Theorem). Let $G$ be a graph and
  suppose $\{a_1,a_2,a_3\}$ is the vertex set of a triangle in
  $G$. Then $G$ is a planar triangulation in which $a_1a_2a_3$ is a
  face if and only if there exists a standard representation
  $R=R(a_1,a_2,a_3)$ of $V=V(G)$ such that  $G=\Sigma_2(R)$.
\end{theorem}

We begin by establishing some basic facts about 
the neighborhood (denoted by $\Gamma$) of $a_1$, and
the vertex $b=\max_{<_1}V\setminus\{a_1\}$ in the graph $\Sigma_2(R)$.
\begin{lemma}\label{sigfacts}
  Let $R=R(a_1,a_2,a_3)$ be a standard representation of $V$, and let 
    the neighbors of $a_{1}$ in $\Sigma_2(R)$ be
    $w_{0}<_{2}w_{1}<_{2}\cdots<_{2}w_{m}<_{2}w_{m+1}$. 
  Then   
\begin{enumerate}
\item $w_{m+1}<_{3}w_{m}<_{3}\cdots<_{3}w_{1}<_{3}w_{0}$, 
\item $w_{0}=a_{3}$ and $w_{m+1}=a_{2}$,
\item the set $S_{i}=\{z\in V:w_{i}<_{2}z<_{2}w_{i+1}\text{ and
  }w_{i+1}<_{3}z<_{3}w_{i}\}$   is empty,
\item $w_{i}w_{i+1}\in \Sigma_2(R)$ for each $i$,
\item the vertex
  $b$ is a neighbor $w_i$ of $a_1$ for some $i$. Moreover,
  if $|V|\geq 4$ then $a_1$
  and $b=w_i$ have exactly two common neighbors, namely $w_{i-1}$ and
  $w_{i+1}$,  
\item if $|V|\geq 4$ then every
  $z\in\Gamma(b)\setminus\{a_1,w_{i-1},w_{i+1}\}$ satisfies
  $w_{i}<_{2}z<_{2}w_{i+1}$ and $w_{i}<_{3}z<_{3}w_{i-1}$.
\end{enumerate}
\end{lemma}
\begin{proof}
For (1), suppose $j<i$. Since $a_1w_i\in\Sigma_2(R)$ we know there is
an
order $k$ in which $a_{1},w_{i}<_{k}w_{j}$. But $k\neq 1$ since $a_1$
is maximum in $<_1$, and $k\neq 2$ by assumption. Thus $k=3$ as
required.

For (2), since $R$ is standard we know each $z\in
  V\setminus\{a_{1},a_{2}\}$ satisfies $a_{1},a_{2}<_{3}z$, so
  $a_1a_2\in\Sigma_2(R)$ and thus $w_{m+1}=a_2$. Similarly
  $a_1a_3\in\Sigma_2(R)$, and so $w_{0}=a_3$ since by (1) we know
  $w_{0}$ is the neighbor of $a_1$ that is highest in $<_3$. 

For (3), suppose on the contary that for some $i$ we have
$y_{\star}=\min_{<_{2}}S_{i}$. Then since $y_{\star}a_{1}$ is not an
  edge of $\Sigma_2(R)$, and $R$ is standard, there exists a
  vertex $z\notin\{y_{\star},a_1\}$ such that $z<_{2}y_{\star}$ and
  $z<_{3}y_{\star}$. Since $w_{i}a_{1}\in\Sigma_2(R)$ we must have
  $w_{i}<_{2}z$. Similarly, we have $w_{i+1}<_{3}z$. This
  contradicts the minimality of $y_{\star}$, since $w_{i+1}<_3 z<_3
  y_{\star}<_3 w_i$ 
  and $w_i<_2 z<_2 y_{\star}<_2 w_{i+1}$. This shows
  $S_i=\emptyset$.

For (4), suppose on the contrary that there exists $y\in
  V\setminus\{w_{i},w_{i+1}\}$ so that for each $k$, 
  either $y<_{k}w_{i}$ or
  $y<_{k}w_{i+1}$. By (1) this implies in particular that
  $y<_{2}w_{i+1}$ and $y<_{3}w_{i}$. But since
  $a_{1}w_{i}\in\Sigma_2(R)$ and $a_{1}w_{i+1}\in\Sigma_2(R)$, we must
  have   $w_{i}<_{2}y$ and $w_{i+1}<_{3}y$, implying $y\in S_i$. This
  contradicts (3).

For (5), let $x\in V(G)\setminus\{a_{1},b\}$. Since $R$ represents $V$
we
know that $b<_jx$ for some $j$, and by definition of $b$ we have
$j\not=1$. Thus, since $R$ is standard, we conclude $a_1<_jb$ also, and
hence $a_1b\in\Sigma_2(R)$, so
$b=w_i$ for some $i$. Provided $|V|\geq 4$, since $R$ is standard we know
$b\notin\{a_2,a_3\}$, and so $0<i<m+1$ by (2). Therefore by (4) we
know $w_{i-1}$ and 
$w_{i+1}$ are distinct common neighbors of $a_1$ and $b$.  
Now suppose on the contrary that
$w_j\notin\{w_{i-1},w_{i+1}\}$ is a third common neighbor of $a_1$
and
$b=w_i$. Suppose $j<i-1$. Then $w_{i-1}<_3w_{j}$, so since
$w_{i-1}<_2w_{i}$, and $w_{i-1}<_1w_i$ by definition of $b=w_i$, we
see that $w_{i-1}$ is not above both of $w_i$ and $w_{j}$ in any
order. This
contradicts the fact  that $w_iw_{j}\in\Sigma_2(R)$. Similarly we find
a contradiction if $j>i+1$. 

For (6), let $z$ be a neighbor of $b=w_{i}$ different from $a_1$, and
suppose on the contrary that $z<_{2}w_{i}$. Again since $|V|\geq 4$ we
have $0<i<m+1$. By (4) we know
$w_{i}w_{i-1}\in\Sigma_2(R)$, so we must have
$w_{i}<_{3}w_{i-1}<_{3}z$. This implies there is no order in which
$w_{i-1}$ is above $z$ and $w_i$, contradicting the fact that
$w_{i}z\in\Sigma_2(R)$. Furthermore, we cannot have $w_{i+1}<_{2}z$,
since in that case $w_{i+1}$ would not be above both $w_{i}$ and $z$ in any order. 
Similarly, we can show that $w_{i}<_{3}z<_{3}w_{i-1}$ as required.
\end{proof}
\end{section}

\begin{section}{Represented graphs are triangulations}
In this section we prove that if $R=R(a_1,a_2,a_3)$ is a standard
representation of 
$V$ then $\Sigma_2(R)$ is a planar triangulation containing
$a_1a_2a_3$ as a face. Our approach will
use induction on $|V|$. Suppose $|V|\geq 4$. For
$b=\max_{<_1}V\setminus\{a_1\}$, we 
define a set $R'$ of three linear orders of the set
$V'=V\setminus\{b\}$ by
suppressing $b$, in other words we let
$<_i'=<_i\vert_{V\setminus\{b\}}$ and set
$R'=\{<_1',<_2',<_3'\}$. Then it is clear from the definitions that
$R'$ is a standard representation of $V'$. 
\begin{lemma}\label{contract} 
With the above definitions, the graph
$\Sigma_2(R')$ is the graph $H$ obtained from $\Sigma_2(R)$ by
 contracting the edge $a_{1}b$ and labelling the new vertex
 $a_1$. 
\end{lemma}
\begin{proof}
First we show that every edge of $H$ is an edge of
$\Sigma_2(R')$. Fix $st\in H$, and suppose $z\in V'\setminus\{s,t\}$. If
$st\in\Sigma_2(R)$, then $s,t<_j z$ for some $j\in\{1,2,3\}$. Thus
$s,t<_j' z$, implying $st\in\Sigma_2(R')$ as required. Otherwise
$st\notin\Sigma_2(R)$, in which case $a_1\in\{s,t\}$, say $s=a_1$, and
$t$ is a 
neighbor of $b$ in $\Sigma_2(R)$. Thus since $bt\in\Sigma_2(R)$,
there exists $j\in\{1,2,3\}$ so that $b,t<_j z$. Since $z\not=a_{1}$
we know $j\not= 1$, as there is no element $z\neq a_1$ which satisfies
	  $b<_{1}z$.
      Since the representation is standard, $a_1<_{j}b$ if $j\neq
	  1$. So
      we have that $a_1,t <_j' z$, showing $st\in\Sigma_2(R')$.

Now we show that every edge $st\in\Sigma_2(R')$ is an edge of $H$. If
$st\in\Sigma_2(R)$ then $s,t\neq b$ so $st\in H$. If
$st\notin\Sigma_2(R)$ then some element of $V$ is not above $s$ and $t$ in
any order. Since $st\in\Sigma_2(R')$ the only possibility for this
element is $b$. This implies $a_1\in \{s,t\}$, say $s=a_{1}$, and that
$b<_{2}t$ and $b<_{3}t$. By
definition of contraction, it suffices to prove that $t$ is a
neighbor of $b$ in $\Sigma_2(R)$. Let $y\in V\setminus\{b,t\}$. If
$y=a_{1}$, then $b,t<_{1}y$. If $y\neq a_{1}$, then since
$a_{1}t\in\Sigma_2(R')$ we know that $a_{1},t<'_{2}y$ or
$a_{1},t<'_{3}y$. This implies $a_{1},t<_{2}y$  or $a_{1},t<_{3}y$.
Therefore $b<_2 t<_{2}y$ or
  $b<_3 t<_{3}y$. This shows $bt\in \Sigma_2(R)$, and the proof is
complete. 
\end{proof}

We are now ready to prove the main result of this section.
\begin{theorem}\label{repisplanar} Let $V$ be a set with $|V|\geq 3$ and let
$R=R(a_1,a_2,a_3)$ be a standard representation of $V$. Then
  $\Sigma_2(R)$ is a planar triangulation containing $a_1a_2a_3$ as a
  face.
\end{theorem}
\begin{proof}
We will proceed by induction on $|V|$. The statement
  holds for $|V|=3$ and $4$, so assume $|V|\geq 5$ and that the
  statement is true for sets of size less than $|V|$. As usual we
  denote
  by $b$ the second-largest element of $<_1$, and by
  $w_0,\ldots,w_{m+1}$ the neighbors of $a_1$ as in Lemma~
  \ref{sigfacts}. Then by Lemma~\ref{sigfacts}(5) $b=w_i$ for some
  $0<i<m+1$.

Let $R'$ be the representation of $V'=V\setminus \{b\}$ obtained by
suppressing $b$. Let
  $z_{0},z_{1},\ldots,z_{d+1}$ be the neighbors of $b$ in
$\Sigma_2(R)$ that are different from $a_{1}$. (Note that
  $d=0$ is possible.) By
Lemma~\ref{sigfacts}(5) we may assume without loss of
  generality that $z_{0}=w_{i-1}$ and $z_{d+1}=w_{i+1}$, and that
  $\{z_1,\ldots,z_d\}\cap\{w_0,\ldots,w_{m+1}\}=\emptyset$. By (6) we
  may assume that
  $w_{i}<_2z_1<_2\ldots<_2z_{d}<_2z_{d+1}=w_{i+1}$. Then by
  Lemma~\ref{contract} the 
  neighborhood of $a_1$ in $\Sigma_2(R')$ is
  $w_0<_2\ldots<_2w_{i-1}<_2z_1<_2\ldots<_2z_{d}<_2w_{i+1}<_2\ldots<_2
  w_{m+1}$. By Lemma~\ref{sigfacts}(4) applied to $R'$ we know
  that consecutive 
  elements in this list are joined by an edge in $\Sigma(R')$, so
  since $w_0w_{m+1}=a_3a_2$ is an edge of $\Sigma_2(R')$, this forms a
  cycle in $\Sigma(R')$ whose vertex set is $\Gamma(a_1)$.

By the induction hypothesis $\Sigma_2(R')$ is a planar triangulation
with $a_1a_2a_3$ as a face, so for convenience let us fix a planar
drawing for which $a_1a_2a_3$ is the outer face. Since in a planar
triangulation there is a unique Hamilton cycle in the subgraph induced
by $\Gamma(v)$ for any vertex $v$ (namely the cycle whose edges are
$\{xy:xyv\hbox{ is a face}\}$), we have a
drawing as shown in Figure \ref{fig}(b). 
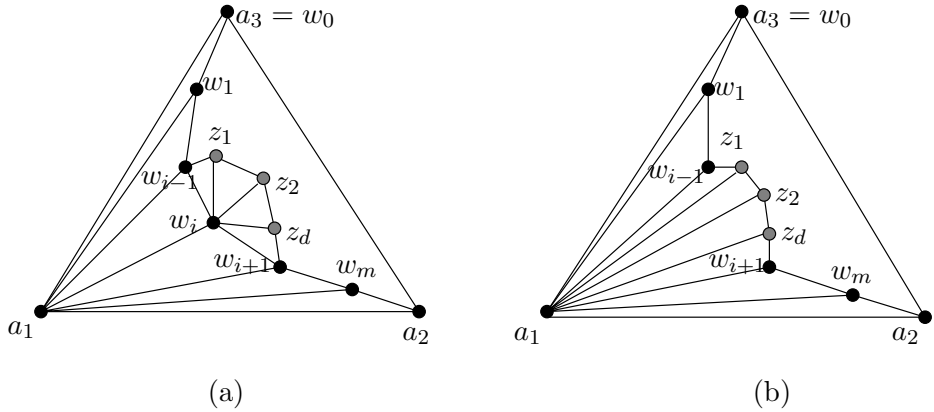
\begin{figure}[hbtp]
\begin{center}
%\parbox{.25\textwidth}{
\input{sig3.pstex_t}
\caption{The neighborhood of $a_1$}
\label{fig}
%}
\end{center}
\end{figure}
We obtain a new planar drawing from this drawing of $\Sigma_2(R')$ by
removing all edges $a_1z_j$ 
for $1\leq j\leq d$, adding a new vertex $w_i=b$ inside the region
bounded by the cycle $a_1z_{0}\ldots z_{d+1}$, and joining $w_i$ to all
of $a_1,z_{0},\ldots,z_{d+1}$ (see Figure \ref{fig}(a)). This is a
planar drawing of $\Sigma_2(R)$, because by Lemma~\ref{contract} every
edge of $\Sigma_2(R)$ that is not an edge of $\Sigma_2(R')$ is
incident to $b$, and the set of edges of $\Sigma_2(R)$ incident to $b$
is precisely $\{ba_1,bz_{0},\ldots,bz_{d+1}\}$. Finally, we note that
since
$\Sigma_2(R')$ is a triangulation and $\Sigma_2(R)$ is a planar
graph with one more vertex and three more
edges, we see $\Sigma_2(R)$ is a planar triangulation, and it has $a_1a_2a_3$
as a face. Thus by induction the proof is complete.
\end{proof}

\end{section}
\begin{section}{Triangulations are represented graphs}
Our aim in this section is to prove the other implication of Theorem
\ref{sch}, as follows.

\begin{theorem} Let $G$ be a planar triangulation and let $a_1a_2a_3$ be a
  face in $G$. Then there exists a standard representation
  $R=R(a_1,a_2,a_3)$  of $V=V(G)$ such that $G=\Sigma_2(R)$.
\end{theorem}

\begin{proof}
Again we will use induction on $|V|$. If $|V|=3$ or $4$ then the result is
true so we assume $|V|\geq 5$. Consider a planar drawing of $G$ for
which $a_1a_2a_3$ is the outer face. Then the neighborhood of $a_1$
is a cycle $a_3=w_0w_1\ldots w_mw_{m+1}=a_2$. Since $G$ is a triangulation,
there exists $w_i$ with $1\leq i\leq m$ such
that $a_1$ and $w_i$ have exactly two common neighbors, namely
$w_{i-1}$ and $w_{i+1}$. (If the cycle $C_0=w_0w_1\ldots w_mw_{m+1}$
is chordless we may choose $w_i=w_1$, otherwise we may choose $w_i$
such that $w_{i-1}w_j$ is a shortest chord of $C_0$, where $j>i$.)
Then the neighbors of $w_i$ in $G$ form a cycle 
$a_1z_0z_1\ldots z_dz_{d+1}$ where $z_0=w_{i-1}$ and $z_{d+1}=w_{i+1}$
(see Figure \ref{fig}(a)).

Let $G'$ be the planar triangulation obtained from $G$ by contracting
the edge $a_1w_i$ and giving the resulting vertex the label
$a_1$. Then the neighborhood of $a_1$ in $G'$ 
is the cycle $a_3=w_0w_1\ldots w_{i-1}z_1\ldots z_dw_{i+1}\ldots
w_mw_{m+1}=a_2$ (see Figure \ref{fig}(b)). 
Since $a_1a_2a_3$ is a face of $G'$, by induction there
exists a standard representation 
  $R'=R'(a_1,a_2,a_3)$  of $V'=V\setminus\{w_i\}$ such that
$G'=\Sigma_2(R')$.   

We claim that $w_0<'_2\ldots<'_2 w_{i-1}<'_2z_1<'_2\ldots<'_2
z_d<'_2w_{i+1}<'_2\ldots <'_2w_{m+1}$ in $R'$. To see this, observe
that by Lemma~\ref{sigfacts}(2) and (4) applied to $R'$, there is a
cycle $C$ with vertex set $\Gamma(a_1)$ in $\Sigma_2(R')$, that
contains the edge $a_2a_3$, and the vertices in this cycle appear in
increasing order in $<'_2$ starting with $w_0=a_3$ and ending with
$w_{m+1}=a_2$. Then since $\Sigma_2(R')$ is a planar triangulation, by
uniqueness of this cycle we must have $C= w_0w_1\ldots
w_{i-1}z_1\ldots z_dw_{i+1}\ldots w_mw_{m+1}$. This proves our claim.

By Lemma~\ref{sigfacts}(1) applied to $R'$ we find
$w_{m+1}<'_{3}\cdots<'_{3}w_{i+1}<'_{3}z_d<'_3\cdots<'_3z_1<'_3w_{i-1}<'_3\cdots<_3w_{0}$.
We define linear orders $R=\{<_1,<_2,<_3\}$ of $V$ from
$R'=\{<'_1,<'_2,<'_3\}$ of $V'$ as follows: we place 
$w_i$ just below $a_1$ in $<'_1$ to form $<_1$. To form $<_2$ we place
$w_i$ just above $w_{i-1}$ in $<'_2$, and we form $<_3$ by placing $w_i$
just above $w_{i+1}$ in $<'_3$.  

Now to complete the proof we verify that $G=\Sigma_2(R)$. We observe
that $R=R(a_1,a_2,a_3)$ is standard by construction. Since
$R'$ represents $V'$, to check that $R$
represents $V$ we just need to verify that
$w_i$ occurs above every $y$ in some order and below every $y$ in some
order. For $y=a_1$ this is immediate as $R$ is standard. For
$y\not=a_1$ we know that $y<_1w_i$. Suppose on the contrary that
$y<_2w_i$ and $y<_3w_i$. Then by our construction $y\leq_2w_{i-1}$ and
$y\leq_3w_{i+1}$. Then $y$ is not above both $a_1$ and $w_{i-1}$, or above
both $a_1$ and $w_{i+1}$, in any order 
$<'_1,<'_2,<'_3$, contradicting the fact that $a_1w_{i-1}$ and
$a_1w_{i+1}$  are edges of $\Sigma_2(R')$. Thus $R$ represents $V$.

Let $uv$ be an edge of $G$. If $uv\in G'=\Sigma_2(R')$ then every
element of $V'=V\setminus\{w_i,u,v\}$ occurs above $u$ and $v$ in some order
$<'_k$ and therefore also in $<_k$. Moreover $w_i$ occurs above $u$
and $v$
in $<_1$ by construction, unless one of them is $a_1$, say $u=a_1$. Since $a_1v \in
G'=\Sigma_2(R')$ we know $v=w_j$ for some $j$. If $j\leq i-1$ then
$a_1,v<_2w_i$ and if $j\geq i+1$ then $a_1,v<_3w_i$. Therefore $uv$ is
an edge of $\Sigma_2(R)$. If $uv\notin G'=\Sigma_2(R')$ then by
definition of contraction we have say $u=w_i$, and 
$v\in\{a_1,z_0,\ldots,z_{d+1}\}$. Now $a_1w_i\in\Sigma_2(R)$ because
$R$ represents $V$, and so $w_i<_2y$ or $w_i<_3y$ for every $y\in
V\setminus\{a_1,w_i\}$. To see that $w_iz_j\in\Sigma_2(R)$, observe
that since $a_1z_j\in\Sigma_2(R')$ by definition of contraction, every
element $y$ of $V\setminus\{w_i,z_j,a_1\}$ occurs above $a_1$ and $z_j$ in some order
$<'_k$. Then $k\in\{2,3\}$, and so also $w_i,z_j<_ky$ by our placement
of $w_i$. Finally note that $w_i,z_j<_1a_1$. Therefore $w_iz_j\in\Sigma_2(R)$. 

Thus we have shown that every edge of $G$ is in $\Sigma_2(R)$. By
Theorem~\ref{repisplanar} we know that $\Sigma_2(R)$ is a planar
triangulation, and since $G$ is also a planar triangulation we conclude
$G=\Sigma_2(R)$. This completes our proof.
\end{proof}

We end with the remark that there are other properties of planar
triangulations $G$ that are nicely captured by properties of a standard
representation $R$ of $V=V(G)$. For example, analogously to the graph
$\Sigma_2(R)$  one can define a 3-uniform hypergraph $\Sigma_3(R)$ with vertex
set $V$ by letting $xyw$ be an edge of
$\Sigma_3(R)$ 
if and only if for every $z\in V$ there exists $i\in\{1,2,3\}$ such
that $z\geq_ix$, $z\geq_iy$ and $z\geq_iw$.
It is quite easy to show that if $R=R(a_1,a_2,a_3)$ is a standard
representation of a vertex set $V$, and $G$ is an embedding of
$\Sigma_2(R)$ in which $a_1a_2a_3$ is the outer face, then
$\Sigma_3(R)=F(G)\setminus\{a_1a_2a_3\}$, where $F(G)$ denotes the set
of faces of $G$. 
 
\end{section}
\begin{section}*{Acknowledgments}
The authors are grateful to Tom Trotter for helpful discussions.
\end{section}

\end{document}

%% file: sig3.pstex_t
\begin{picture}(0,0)%
\epsfig{file=sig3.pstex}%
\end{picture}%
\setlength{\unitlength}{3947sp}%
\begingroup\makeatletter\ifx\SetFigFont\undefined%
\gdef\SetFigFont#1#2#3#4#5{%
  \reset@font\fontsize{#1}{#2pt}%
  \fontfamily{#3}\fontseries{#4}\fontshape{#5}%
  \selectfont}%
\fi\endgroup%
\begin{picture}(5810,2543)(3601,-5426)
\put(4858,-5426){\makebox(0,0)[lb]{\smash{\SetFigFont{11}{13.2}{\rmdefault}{\mddefault}{\updefault}(a)}}}
\put(8281,-5426){\makebox(0,0)[lb]{\smash{\SetFigFont{11}{13.2}{\rmdefault}{\mddefault}{\updefault}(b)}}}
\put(4439,-4064){\makebox(0,0)[lb]{\smash{\SetFigFont{11}{13.2}{\rmdefault}{\mddefault}{\updefault}$w_{i-1}$}}}
\put(5278,-4099){\makebox(0,0)[lb]{\smash{\SetFigFont{11}{13.2}{\rmdefault}{\mddefault}{\updefault}$z_2$}}}
\put(5347,-4413){\makebox(0,0)[lb]{\smash{\SetFigFont{11}{13.2}{\rmdefault}{\mddefault}{\updefault}$z_d$}}}
\put(7618,-4029){\makebox(0,0)[lb]{\smash{\SetFigFont{11}{13.2}{\rmdefault}{\mddefault}{\updefault}$w_{i-1}$}}}
\put(8421,-4168){\makebox(0,0)[lb]{\smash{\SetFigFont{11}{13.2}{\rmdefault}{\mddefault}{\updefault}$z_2$}}}
\put(3601,-5007){\makebox(0,0)[lb]{\smash{\SetFigFont{11}{13.2}{\rmdefault}{\mddefault}{\updefault}$a_1$}}}
\put(6081,-5042){\makebox(0,0)[lb]{\smash{\SetFigFont{11}{13.2}{\rmdefault}{\mddefault}{\updefault}$a_2$}}}
\put(9155,-5042){\makebox(0,0)[lb]{\smash{\SetFigFont{11}{13.2}{\rmdefault}{\mddefault}{\updefault}$a_2$}}}
\put(6780,-5042){\makebox(0,0)[lb]{\smash{\SetFigFont{11}{13.2}{\rmdefault}{\mddefault}{\updefault}$a_1$}}}
\put(8072,-3819){\makebox(0,0)[lb]{\smash{\SetFigFont{11}{13.2}{\rmdefault}{\mddefault}{\updefault}$z_1$}}}
\put(8281,-3051){\makebox(0,0)[lb]{\smash{\SetFigFont{11}{13.2}{\rmdefault}{\mddefault}{\updefault}$a_3=w_0$}}}
\put(8037,-3505){\makebox(0,0)[lb]{\smash{\SetFigFont{11}{13.2}{\rmdefault}{\mddefault}{\updefault}$w_1$}}}
\put(8771,-4658){\makebox(0,0)[lb]{\smash{\SetFigFont{11}{13.2}{\rmdefault}{\mddefault}{\updefault}$w_m$}}}
\put(8457,-4413){\makebox(0,0)[lb]{\smash{\SetFigFont{11}{13.2}{\rmdefault}{\mddefault}{\updefault}$z_d$}}}
\put(8002,-4587){\makebox(0,0)[lb]{\smash{\SetFigFont{11}{13.2}{\rmdefault}{\mddefault}{\updefault}$w_{i+1}$}}}
\put(5033,-3051){\makebox(0,0)[lb]{\smash{\SetFigFont{11}{13.2}{\rmdefault}{\mddefault}{\updefault}$a_3=w_0$}}}
\put(4614,-4343){\makebox(0,0)[lb]{\smash{\SetFigFont{11}{13.2}{\rmdefault}{\mddefault}{\updefault}$w_i$}}}
\put(4823,-3470){\makebox(0,0)[lb]{\smash{\SetFigFont{11}{13.2}{\rmdefault}{\mddefault}{\updefault}$w_1$}}}
\put(4858,-3785){\makebox(0,0)[lb]{\smash{\SetFigFont{11}{13.2}{\rmdefault}{\mddefault}{\updefault}$z_1$}}}
\put(4894,-4587){\makebox(0,0)[lb]{\smash{\SetFigFont{11}{13.2}{\rmdefault}{\mddefault}{\updefault}$w_{i+1}$}}}
\put(5662,-4623){\makebox(0,0)[lb]{\smash{\SetFigFont{11}{13.2}{\rmdefault}{\mddefault}{\updefault}$w_m$}}}
\end{picture}

%% file: main.bbl
\begin{thebibliography}{00}

\bibitem{BD} L. Babai, D. Duffus, Dimension and automorphism groups of
  lattices, \emph{Algebra Universalis} 12 (1981), 279--289.

\bibitem{BT1} G. Brightwell, W.T. Trotter, The order dimension of
  convex polytopes, \emph{SIAM J. Discrete Mathematics} 6 (1993),
  230--245.

\bibitem{BT2} G. Brightwell, W.T. Trotter, The order dimension of
  planar maps, \emph{SIAM J. Discrete Mathematics} 6 (1997),
  515--528.

\bibitem{CFL} L. Castelli Aleardi, E. Fusy, T. Lewiner, Schnyder woods
for higher genus triangulated surfaces, with applications to encoding.
\emph{Discrete Comput. Geom.}  42  (2009),  489--516.  

\bibitem{F1} S. Felsner, Convex drawings of planar graphs and the
  order dimension of 3-polytopes, \emph{Order} 18 (2001), 19--37.  

\bibitem{F2} S. Felsner, Geodesic embeddings and planar graphs,
  \emph{Order} 20 (2003), 135--150.  

\bibitem{FT} S. Felsner, W.T. Trotter, Posets and planar graphs,
  \emph{J. Graph Theory} 49 (2005), 273--284.

\bibitem{FZ} S. Felsner, F. Zickfeld, Schnyder woods and orthogonal
surfaces.  \emph{Discrete Comput. Geom.}  40  (2008), 103--126.  

\bibitem{DPP} H. de Fraysseix, J. Pach, R. Pollack, How to draw a
planar graph on a grid, \emph{Combinatorica} 10 (1990), 41--51. 

\bibitem{pom} P. Ossona de Mendez, Geometric realization of simplicial
  complexes, \emph{Graph Drawing, Lecture Notes in Computer Science} 5
  (1999), 323--332.
\bibitem{S} W. Schnyder, Planar graphs and poset dimension,
  \emph{Order} 5 (1989), 323--343.
    
\end{thebibliography}
